\documentclass{amsart}  
\usepackage[parfill]{parskip}
\usepackage{graphicx, verbatim}
\usepackage{appendix}
\usepackage{amsfonts}
\usepackage{url}
\usepackage{hyperref} 
\hypersetup{backref,pdfpagemode=FullScreen,colorlinks=true}
\usepackage{amsmath}
\usepackage{amssymb}
\usepackage{tikz-cd}
\usepackage{amscd}
\usepackage{color}
\usepackage{amsthm}
\usepackage{bm}
\usepackage{indentfirst}
\usepackage[hmargin=3cm,vmargin=3cm]{geometry}
\numberwithin{equation}{section}

\newtheorem{theorem}[equation]{Theorem} 
\newtheorem{proposition}[equation]{Proposition}
 
\newtheorem{lemma}[equation]{Lemma} 
 
\newtheorem{corollary}[equation]{Corollary}

\newtheorem{definition}[equation]{Definition}
\theoremstyle{definition}

\newtheorem{terminology}{Terminology}

\theoremstyle{remark}

\newtheorem{remark}[equation]{Remark}
\newtheorem{example}{Example}
\newtheorem{question}{Question}

\DeclareMathOperator {\mult} {mult}

\DeclareMathOperator {\energy} {energy}

\DeclareMathOperator{\image}{\mathrm{image}}

\begin{document}
\title {Untwisted Gromov-Witten invariants of Riemann-Finsler
manifolds}
\author{Yasha Savelyev}
\email{yasha.savelyev@gmail.com}
\address{Faculty of Science, University of Colima, Mexico}

\begin{abstract} 
We define a $\mathbb{Q}$-valued deformation invariant of certain
complete Riemann-Finsler manifolds, in particular of complete
Riemannian manifolds with non positive sectional curvature. It is proved that every rational number is the value of this invariant for
some compact Riemannian manifold.
We use this to find the first and mostly sharp
generalizations, to non-compact products and fibrations, of Preissman's theorem on
non-existence of negative sectional curvature metrics on
compact products. For example, $\Sigma \times T ^{n}$ admits
a metric of negative sectional curvature, where
$\Sigma$ is a non-compact possibly infinite type
surface, if and only if $\Sigma$ has genus zero. We also give novel estimates on
counts of closed geodesics with restrictions on 
multiplicity. Along the way, we also prove that sky catastrophes of smooth dynamical
systems are not geodesible by a certain class of forward complete
Riemann-Finsler metrics, in particular by complete Riemannian metrics with
non-positive sectional curvature. This partially answers
a question of Fuller and gives important examples for our
theory here.
\end{abstract}
\maketitle
\section{Introduction}
We will define certain rational number valued deformation
invariants for certain complete Riemann-Finsler manifolds,
and in particular for all complete
Riemannian manifolds with non-positive sectional curvature.  These invariants can be
directly interpreted as the untwisted part of certain elliptic Gromov-Witten invariants in an associated lcs manifold,
~\cite{cite_SavelyevEllipticCurvesLcs}. The twist is coming by
way of metric isometries. In the more basic, untwisted
setting here, we reduce the invariants
to counts of \textbf{\emph{geodesic strings}}  (equivalence classes of closed, unit speed geodesics up to reparametrization $S ^{1}$ action), via the Fuller index
~\cite{cite_FullerIndex}. 

The latter count in modern
terms is just an orbifold count of points of virtual
dimension 0 Kuranishi spaces, corresponding to spaces of
geodesic strings in a fixed free homotopy class. But
Fuller's construction makes the latter elementary and
geometric.

Indeed, one of  the main ideas of this note is that one
can make geodesic string counts into a global metric deformation invariant, provided
we work with some metric curvature restrictions. Studying
connections with metric geometry was suggested by Fuller
himself in the 1960's. We now discuss some applications.

Recall that a now classical theorem of Preissman
~\cite{cite_PreissmanNegativeCurvatureProducts} implies that
there are no non-trivial compact products with negative
sectional curvature. We give the first generalizations of
this to the non-compact case. As a basic example:
\begin{theorem} \label{thm_IntrojOOOOOoTn0}
There is a complete negative sectional curvature metric $g$ on
$M = \Sigma \times T ^{n}$, where $\Sigma$ is a non-compact, possibly infinite type surface, if and only if
$\Sigma$ has genus zero.  
\end{theorem}
When the genus is zero, existence of such a $g$ is well
known. For example, if $\Sigma$ is the
infinite cylinder $g$ can be given as the warped product
metric. 

If in addition we require that $g$ be finite volume and 
that $M$ be the interior of a compact manifold with
boundary, one might try to prove that case of the theorem by
the following strategy. Assume such a $g$ exists. Use
``doubling'' to get a negatively curved metric on the
compact double.  Apply Preissman's theorem to get
a contradiction.
\begin{remark} \label{rem_}
We cannot just double, one needs to do a surgery on
ends, glue and deform to a smooth negatively curved metric. For
example, doubling of cusped hyperbolic surfaces is already complex.
\end{remark}
If $\Sigma$ is infinite type, for example the standard infinite genus
surface, the idea above totally fails. Our argument does not use
Preissman's theorem, and is inherently based on functional
analysis and infinite dimensional Morse theory in particular. 

The only if part of the above theorem is just a very special case of the
following. We denote by $\pi
_{1} ^{inc} (X)$ the set of free homotopy classes of loops
incompressible to the ends, in the sense of Definition
\ref{definition_boundaryincompressible}, (non-constant
classes when $X$ is closed). 
\begin{theorem} \label{corollary_product}
Let $X = Z \times Y$ where $Z,Y $ admit complete 
Riemannian metrics with non-positive sectional curvature,
$Y$ is closed,  $\pi _{1} ^{inc} (Z) \neq 0$, $\chi(Y) \neq
\pm 1$. Then:
\begin{itemize}
\item $X$ does not admit
a complete metric of negative sectional curvature, or
a forward complete Finsler metric with negative flag curvature. 
\item Moreover, $X$ does not even admit a forward complete
Finsler metric with a unique and non-degenerate class $\beta
$ geodesic string for any $\beta \in \pi _{1}  ^{inc} (Z)$.
\end{itemize}
\end{theorem}
This theorem is very close to being sharp, for example the
conclusion of the corollary is false if $Y=S ^1$ and $X
= S ^1 \times \mathbb{R}  $. As $Z = T ^{2} \times
\mathbb{R} ^{} $ admits the warped
product metric $g _{T ^{2}} \times _{e ^{t}} g _{\mathbb{R}
^{} }$ (with respect to the function $f= e ^{t}$ on
$\mathbb{R} ^{} $) where $g _{T ^{2}}, g _{\mathbb{R} ^{} }$
are the flat metrics. This warped product has constant negative sectional curvature $-1$.
So it is essential that not only $\pi _{1} (Z) \neq
0$ but also that there is a class incompressible to the
ends. Of course, $\chi(Y) =1$ is also
obviously essential, otherwise we may take $Y =pt$. The condition $\chi(Y) \neq -1$ is however
not obviously essential. 

We will also
give various generalizations of this result to fibrations.
For fibrations, the most obvious analogue of Preissman's theorem fails even
assuming compactness. In fact,  
every closed 3-manifold $X ^{3}$, for which there is no injection
$\mathbb{Z} ^{2} \to \pi _{1} (X, x _{0})$, and which
fibers over a circle has a hyperbolic structure $g _{h}$,
Thurston~\cite{cite_ThurstonClassificationSurfaceDiffeo}. 

A more direct type of application is to geodesic
counting with constraints on multiplicity. 
This is based
on the idea that our Gromov-Witten type invariant is a
metric deformation invariant provided there is a curvature control.
For a non-constant class $\beta \in \pi _{1} ^{inc} (X)$, we say that a metric $g$ on $X$ is $\beta $-\textbf{\emph{regular}} if all
of its class $\beta $ closed geodesics are non-degenerate
(meaning that the closed orbits of the associated geodesic
flow are non-degenerate, i.e. the associated fixed points
are non-degenerate.) For a geodesic string $o$, $\mult
(o) $ will denote its multiplicity, (the order of the
corresponding isotropy subgroup of $S ^{1}$,  where $S ^{1}$ is acting by 
reparametrization.) And $\operatorname {morse} (o) $ will
denote the Morse index of $o$, (meaning the Morse-Bott index
of the associated critical submanifold of the loop space.)

\begin{definition}\label{definition_Hadamardequivalent}
We say that a pair $g _{0}, g _{1}$ of forward complete
Finsler metrics with
non-positive flag curvature are
\textbf{\emph{Hadamard equivalent}} if there is
a continuous interpolation $\{g _{t}\}$, $t \in [0,1] $,
(with respect to the topology of $C
^{0}$ convergence on compact sets) s.t. each $g _{t}$ has
non-positive flag curvature.
\end{definition}
We avoid maximal generality in what follows, as these are
just sample applications. 
\begin{theorem} \label{thm_estimateIntro}
Suppose that $g _{1}$ is a forward complete, non-positive flag
curvature metric on $X$, Hadamard
equivalent to a forward complete metric of negative flag
curvature. 
Suppose that $\beta \in \pi _{1} ^{inc} (X)$ is a $k$-power,
(Definition \ref{definition_indivisible}). Let $C _{\beta} = 2 \cdot
L _{\beta}$, where $L _{\beta}$ is the length of a class $\beta$, $g _{1}$-geodesic.
For all $\epsilon >0$
sufficiently small, whenever $g'$ is $C
^{0}$ $\epsilon$ close to $g _{1}$, and is $\beta $-regular,
we have:
\begin{equation*} 
\displaystyle \sum_{o \in
\mathcal{O} _{C _{\beta}} (g',\beta)} \frac{(-1) ^{\operatorname {morse}  (o)}}
{\mult (o)} = \frac{1}{k},  
\end {equation*}
where $\mathcal{O} _{C _{\beta}} (g', \beta)$ is the set of class $\beta
$ geodesic strings with $g'$-length less than $C _{\beta}$.
\end{theorem}

We may use this to estimate the number of
closed geodesic strings with constraints on multiplicity and energy. Here is
a sample immediate corollary.
\begin{corollary} \label{cor_estimate}
Let $X$, $g _{0}$ and $g _{1}$ be as above. Suppose that $\beta \in \pi _{1} ^{inc} (X)$ is a $p$-power
for $p$ a prime. 
Then for all $\epsilon >0$ sufficiently small,  whenever
$g'$ is $C ^{0}$ $\epsilon$ close to $g _{1}$, 
and has a class $\beta$ geodesic string with $g'$-length
less than $C _{\beta}$, and with multiplicity
$k \neq p$, then $g'$ has at least two such geodesic strings.
\end{corollary}
The following theorem in particular says that for regular
metrics on $T ^{2}$ sufficiently nearby to the flat metric, fixed class geodesic strings with ``small length'' must come at
least in pairs, corresponding to multiplicity constraints. 
\begin{theorem} \label{thm_IntroTn}
Let $g _{0}$ be the standard flat metric $X=T ^{n}$. 
For all $\epsilon
>0$ sufficiently small, whenever $g'$ is a Finsler metric $C ^{0}$ $\epsilon$
close to $g _{0}$ and is $\beta $-regular, the following
holds.  Suppose $g'$ has a closed geodesic string $o$ s.t. 
\begin{enumerate}
	\item $o$ has class $\beta $.
	\item $p | \mult (o) $, where $p$ is prime.
	\item $g'$-length of $o$ is less than $C _{\beta}$.
\end{enumerate}
Then $g'$ has at least one other geodesic
string satisfying these conditions. 
Moreover, any Finsler metric $g _{1}$ Hadamard equivalent
to $g _{0}$ has the same property as $g _{0}$, above.
\end{theorem}
We will discuss the statements in more detail in what
follows.
\section{Setup and statements} \label{sec_Statement of results}
\begin{terminology} From now on, all our metrics are
Riemann-Finsler (a.k.a. Finsler)
metrics unless specified to be Riemannian, and usually denoted by
just $g$. 
Completeness, always
means forward completeness, and it is an assumption for all
our metrics. Curvature always means sectional
curvature in the Riemannian case and flag curvature in the
Finsler case. Thus we will usually just
say complete metric $g$, for a forward complete
Riemann-Finsler metric.  	A reader may certainly choose to
interpret all metrics as Riemannian metrics, completeness
as standard completeness, and curvature as sectional
curvature. 
\end{terminology}
In what follows $\pi _{1} (X)$ denotes the
set of free homotopy classes of continuous maps $o: S ^{1} \to X$. 

\begin{definition}\label{definition_boundaryincompressible}
Let $X$ be a smooth manifold. Fix an exhaustion by
nested compact sets $\bigcup _{i \in \mathbb{N}} K _{i}
= X$, $K _{i} \supset K _{i-1}$ for all $i \geq 1$.   We say
that a class $\beta \in \pi _{1} (X)$ is {end
compressible} if $\beta $ is in the image of $$inc _{*}: \pi
_{1}(X - K _{i}) \to \pi _{1} (X)$$  for all $i$, where
$inc: X - K _{i} \to X$ is the inclusion map. 
We say that $\beta $ is \textbf{\emph{end incompressible}}
(or incompressible to the ends)
if it is not end compressible.
\end{definition}
Let $\pi _{1} ^{inc} (X)$ denote the
set of such end incompressible classes. When $X$ is
compact, we set $\pi _{1} ^{inc} (X) := \pi _{1} (X)
- const$, where $const$ denotes the set of homotopy classes
of constant loops.

It is easily seen that the above is well defined
(independent of the choice of an exhaustion) and
moreover any homeomorphism $X _{1} \to X _{2}$ of
a pair of manifolds induces a set isomorphism $\pi _{1} ^{inc} (X _{1}) \to \pi _{1} ^{inc} (X _{2}) $.
Denote by $L _{\beta } X$ the class $\beta \in \pi _{1}
^{inc}(X) $ component of the free loop space of $X$, with
its compact open topology. Let $g$ be a complete metric on $X$, and let $S
(g, \beta) \subset L _{\beta } X$ denote the subspace of all
unit speed parametrized, smooth, closed $g$-geodesics in class $ \beta$.
\begin{definition} \label{definition_betataut}
We say that a metric $g$ on $X$ is \textbf{\emph{$\beta
$-taut}} if it is complete and $S (g, \beta )$ is compact. We will say that $g$ is \textbf{\emph{taut}} if
it is $\beta $-taut for each $\beta \in \pi _{1} ^{inc}
(X)$. 
\end{definition}
\begin{lemma} \label{lemma_nonpositiveistaut}
A complete metric $g$ with non-positive curvature satisfies:
\begin{itemize}
	\item All of its closed geodesics are minimizing in their
	free homotopy class.
	\item It is taut.
\end{itemize}
\end{lemma}
\begin{proof} [Proof]
The first part is a standard consequence of the
Cartan-Hadamard theorem. The second part follows by the first part and Lemma \ref{lemma_compact}.
\end{proof}
It should be emphasized that taut metrics form a much
larger class of metrics then just non-positive curvature
metrics. For example any sufficiently $C ^{1}$ small
perturbation of a metric with non-positive curvature will be
taut. (Indeed, this is crucial for the construction of our
invariant.)
Another class of examples comes by way of Lemma
\ref{lemma_betataut} ahead,  these metrics may not be
non-positively curved nor nearby to metrics non-positively
curved. 
\begin{definition}\label{def_tauthomotopy} Let 
$ \beta \in \pi_{1} ^{inc}(X) $, and let
$g _{0}, g _{1}$
be a pair of $\beta $-taut metrics on $X$. 
A \textbf{\emph{$\beta $-taut deformation}} between $g _{0}, g _{1}$, is a continuous (in the topology of $C ^{0}$
convergence on compact sets) family $\{g _{t}\}$, $t \in
[0,1]$ of complete metrics on $X$, 
s.t. $$S (\{g _{t}\}, \beta ) := \{(o,t) \in L _{\beta }X
\times [0,1] \,|\, o \in S (g _{t}, \beta )\}$$ is compact.
We say that $\{g _{t}\}$ is a \textbf{\emph{taut deformation}}  if it is $\beta $-taut for
each $ \beta  \in \pi _{1} ^{inc}(X)$.  The above definitions
of tautness are extended naturally to the case of a smooth fibration $X
\hookrightarrow  P \to [0,1]$, with a smooth fiber-wise family of metrics.
\end{definition}
A useful criterion for $\beta $-tautness is the following.
\begin{theorem} \label{thm:notgeodesible} Let $\{g _{t}\}
_{t \in [0,1]}$ be a continuous family of complete
metrics on $X$.
Suppose that:
\begin{equation*}
\sup _{t} |\max _{o \in S  (g _{t}, \beta )} l _{g _{t}} (o) -  \min
_{o \in S  (g _{t}, \beta )} l _{g _{t}} (o)| < \infty, 
\end{equation*}
where $l _{g _{t}}$ is the length functional
with respect to $g _{t}$,
then $\{g _{t}\}$ is $\beta $-taut.
It follows that sky catastrophes of vector fields on closed manifolds are not geodesible by
metrics all of whose geodesics are minimal, Appendix
\ref{appendix_bluesky}. 
\end{theorem}
For example, the hypothesis is trivially satisfied if $g
_{t}$ have the property that all their class $\beta
$ closed geodesics are minimal in their homotopy  class.
\begin{corollary} \label{cor_}
If $g _{t}$, $t \in [0,1]$ have non-positive 
curvature then $\{g _{t}\}$ is taut.
\end{corollary}
\begin{proof} [Proof]
This follows by the theorem and by Lemma
\ref{lemma_nonpositiveistaut}. 
\end{proof}

Fuller at the end of ~\cite{cite_FullerIndex} has asked for
any metric conditions on vector fields to rule out sky
catastrophes, see Appendix \ref{appendix_bluesky}. By the
above, non-positivity of curvature is one such
condition. So this is a partial answer to his question.

Note that if sky catastrophes
were never geodesible (or at least if geodesible sky
catastrophes are necessarily unstable, as was conjectured in ~\cite{cite_SavelyevFuller}) then the geodesible Seifert conjecture would follow, by the main
result of ~\cite{cite_SavelyevFuller}. Hence, this is a subtle
question. 
The qualitative structure of such
potential geodesible or Reeb sky catastrophes is somewhat
understood, ~\cite[Theorem 1.10]{cite_SavelyevFuller}.
But
this does not greatly aid constructing potential
examples, which must be topologically very complex, (there
are necessarily infinitely many suitably synchronized bifurcation
events).  No results prior to the theorem above
are known to me aside from those    mentioned by Fuller
himself in ~\cite{cite_FullerIndex}. 

\subsection{The geodesic counting invariant $F$} \label{sec_Definition of the invariant F}
Let $\mathcal{G} (X)$ be the set of equivalence classes of
taut metrics $g$, where $g _{0}$ is equivalent to $g _{1}$ whenever there is
a taut deformation between them.  We may denote an
equivalence class by its representative $g$ by a slight abuse
of notation. 

\begin{theorem} \label{thm:invariantF}
For each manifold $X$ there is a natural, non-trivial
functional: $$\operatorname {F}: \mathcal{G} (X) \times \pi
_{1} ^{inc} (X) \to \mathbb{Q}. $$ 
\end{theorem}
The value $\operatorname {F} (g, \beta )  $ is a certain
weighted count of the set of closed  $g$-geodesic
strings in class $\beta $.  
But one must take care of exactly how to count, as in general this set should be
understood as an orbifold or rather a Kuranishi space (as
introduced by Fakaya-Ono ~\cite{cite_FukayaOnoArnoldandGW}), hence this is why $\operatorname {F} $ is
$\mathbb{Q} $ valued. 
In the special case when $g$ is $\beta $-regular we have the
following formula:
\begin{equation*}
F (g, \beta ) = \displaystyle \sum_{o \in
\mathcal{O} (g,\beta)} \frac{(-1) ^{\operatorname {morse}  (o)}}
{\mult (o)}, 
\end{equation*}
where $\operatorname {morse} (o)$ and $\mult (o) $ are as in
the Introduction.

\begin{question} \label{question_tauthomotopy} Do there
exist a pair of taut metrics $g _{1}, g _{2}$ on a manifold
$X$ which are not taut homotopic? Or more generally not
homotopic via a family of metrics without a sky catastrophe?
\end{question}
Probably both possibilities are interesting. If the answer
is `no' then we can obtain much sharper applications of
Theorem \ref{thm:invariantF}, particularly in the setup of
~\cite{cite_SavelyevEllipticCurvesLcs}.  Moreover, in this
case $F$ becomes a topological invariant, which is
a priori unrelated to any classical topological invariant.  
On the other hand:
\begin{corollary} \label{proposition_existsSky} Suppose
for a pair $g _{1}, g _{2}$ of $\beta $-taut metrics on $X$:
\begin{equation*}
F (g _{1}, \beta ) \neq F (g _{2}, \beta),
\end{equation*}
then any path $\{g _{t}\}$, connecting $g _{0},
g _{1}$, is not $\beta$-taut and in fact has a sky
catastrophe. So that if such a pair $g _{1}, g _{2}$ exists
the conjecture of ~\cite{cite_SavelyevFuller} would be
disproved.
\end{corollary}
\begin{proof} [Proof]
The fact that any connecting $\{g _{t}\}$ is not $\beta $-taut is just a direct
corollary of the theorem above. The fact that $\{g _{t}\}$
has a sky catastrophe follows by \cite[Theorem 3.2
]{cite_SavelyevFuller}.
\end{proof}
If the answer to the question above is `yes', then we
should be able to use the above corollary to find non-positively
curved metrics which cannot be joined by a continuous family of
non-positively curved metrics.
(Apparently, existence of such metrics is open.)

\begin{definition}\label{definition_indivisible}
Let $\beta \in \pi _{1} (X)$. For any based point $x _{0}
\in \image \beta \subset X$ (for $\image \beta $ the image
of some representative of $\beta $)  there is 
a naturally determined element $\beta _{x _{0}} \in \pi _{1} (X, x _{0})$ well defined up to
an inner automorphism, (concatenate a representative of
$\beta $ with a path from $x _{0}$ to a point in $\image \beta$). 
\begin{itemize}
 \item We say that a class $\beta \in \pi _{1} (X, x _{0})$ is 
\textbf{\emph{not a power}} if whenever $\beta = \alpha
^{k}$ for some $\alpha, k > 0$ then $k=1$. 
\item We say that a class $\beta \in \pi _{1} (X, x _{0})$ is 
a $k$-\textbf{\emph{power}} if $\beta = \alpha ^{k}$ for
some $\alpha $ which is not a $n$-power for any $n$. 
\item We say that $\beta$ is \textbf{\emph{atomic}} if it is
a $k$-power for some $k$. 
\footnote{As we are working with non-compact
manifold, we may in principle have non atomic classes.}
\item We say that $\beta \in \pi _{1} (X)$ is not
a power, respectively is a $k$-power, respectively is atomic
if for any $x _{0}$ as above, $\beta  _{x _{0}}$ is not
a power, respectively is a $k$-power, respectively is
atomic.
\end{itemize}
\end{definition}
Note that if a class $\beta \in \pi _{1} ^{inc} (X)$ is not
a power then any representative of this class is not
multiply covered, but the converse generally does not hold.
\begin{example} \label{exm_negativecurvatureFuller} Let $g$ be
a Riemannian metric with negative sectional curvature on a closed
manifold $X$ and $ \beta
\in \pi _{1} (X)$ a class represented by a multiplicity $n$ closed
geodesic, then 
\begin{equation} \label{equation_frac1n}
F (g, \beta ) = \frac{1}{n}.
\end{equation}
In particular,
if $\beta $ is not a power then $F (g, \beta ) = {1}$. More generally,
\eqref{equation_frac1n} holds whenever $g$ has a unique
and non-degenerate geodesic string in class $\beta $, where
non-degenerate is as in the Introduction.
\end{example}
\begin{theorem} \label{theorem_valuesOfInvariant} Every
rational number has the form $F (g, \beta )$ for some
$\beta$-taut Riemannian $g$ on some compact manifold $X$ and
for some $\beta \in \pi _{1} ^{inc} (X)$.
\end{theorem}

If $\beta \in \pi _{1} ^{inc} (X)$ is not a power, then it
is easy to see that that the reparametrization $S ^{1}$
action on $L _{\beta }X$ is free (see Appendix \ref{appendix:Fuller}), so that  $H ^{S ^{1}} _{*} (L
_{\beta}X, \mathbb{Z}) \simeq H  _{*} (L
_{\beta}X/ {S ^{1}}, \mathbb{Z}) $, where $H ^{S ^{1}} _{*} (L
_{\beta}X, \mathbb{Z})$ denotes the $S ^{1}$-equivariant
homology. Moreover, we have:
\begin{theorem} \label{theorem_Eulercharacteristic}
Suppose that $\beta \in \pi _{1} ^{inc} (X)$ is not a power,
and $X$ admits a $\beta $-taut metric, then $H ^{S ^{1}} _{*} (L
_{\beta}X, \mathbb{Z})$ is finite dimensional.
Denote by $\chi ^{S ^{1}} (L _{\beta} X)$ the Euler characteristic of this homology.
Then for any $\beta $-taut metric $g$ on $X$:
\begin{equation*}
F (g, \beta ) = \chi ^{S ^{1}}(L _{\beta }X).  
\end{equation*}
\end{theorem}
Explicit examples for the theorem above can be found by the proof of Theorem
\ref{theorem_valuesOfInvariant}. For these types of examples
any negative integer may appear as the value of $F (g, \beta
)$. We leave out the details.
\begin{remark} \label{example_indivisible}
If $\beta $ is a power, the idea behind Theorem \ref{theorem_Eulercharacteristic} breaks
down, as the $S ^{1}$-equivariant homology of $L
_{\beta }X$ may then be infinite dimensional even if $X$ admits
a $\beta $-taut $g$. As a trivial example this homology is
already infinite dimensional when $g$ is negatively curved,
and the class $\beta $ geodesic is $k$-covered, as then this
homology is the group homology of $\mathbb{Z} _{k}$.
In particular the connection with the Euler
characteristic a priori breaks down. It is thus an interesting  open problem if the 
functional $F$ remains topological, this is related to question \ref{question_tauthomotopy}.
\end{remark}

\subsection{Applications to existence of negative curvature
metrics} \label{sec_Applications to existence of negative curvature
metrics}
A celebrated theorem of Preissman ~\cite{cite_PreissmanNegativeCurvatureProducts} says that there are no negative sectional curvature metrics on compact products.
Fibration counterexamples to  Preissman's product theorem
certainly exist as mentioned in the Introduction. We are going to give a certain generalization of Preissman's
theorem to fibrations, with possibly non-compact fibers,
also replacing the negative sectional curvature condition by
a significantly weaker condition.

\begin{definition}\label{def:Gsubmersion}
Let $Z \hookrightarrow X \xrightarrow{p} Y$ be a smooth fiber
bundle with $X$ having a $\beta $-taut
Riemannian metric $g$, for $ \beta \in
\pi_{1} ^{inc}(X)$, and let $g _{Y}$ be a metric on $Y$.
Suppose that: 
\begin{enumerate}
	\item The fibers $Z _{y} = p ^{-1} (y)$  are totally $g$-geodesic, for
	closed geodesics in class $\beta $. We denote by $g _{y}$
	the metric $g$ restricted to $Z _{y}$.
	\item The fibers are parallel (the distribution $T
	^{vert}X = \ker p_*$ is parallel along any smooth curve in $X$
	with respect to the Levi-Civita connection of $g$).
	\item For any pair of fibers $(Z _{y_0}, g _{y _{0}})$, $(Z
	_{y_1}, g _{y _{1}})$, and a path $\gamma: [0,1] \to Y$ from $y
	_{0}$ to $y _{1}$ the fiber family $\{(Z _{\gamma (t)},
	g _{\gamma (t)})\}$ furnishes a taut deformation. \label{item_tauthomotopyfibers}
	\item $p$ projects $g$-geodesics to geodesics of $Y,g
	_{Y}$. \label{item_pProjectionOfgeodesics}
\end{enumerate}
We then call $p: X \to Y$ a  \textbf{\emph{$\beta $-taut
fibration}}, with the metrics $g,g _{Y}$ and $g _{Z}$
all possibly implicit. 
\end{definition}
\begin{definition}\label{definition_fiberclass}
For $Z \hookrightarrow X \to Y$ as above, we say that $\beta \in \pi _{1} (X)$ is
a \textbf{\emph{fiber class}}  if it is in the image of the
inclusion $i _{Z}: \pi _{1} (Z) \to \pi _{1} (X) $.
\end{definition}
In the above definition of a taut fibration and the following theorem we need
the auxiliary metric $g$ on $X$ to be Riemannian, and there is no obvious
extension of the theorem to the Riemann-Finsler case.
However, the conclusions of the theorem are for Riemann-Finsler metrics.
\begin{theorem} \label{thm:Fibration}   Let $p: (X,g) \to
(Y, g _{Y})$ be
a $\beta $-taut fibration, where $ \beta \in
\pi_{1} ^{inc}(X)$ is a fiber class.
Suppose further that $Y$ is connected, 
closed, $\chi(Y) \neq \pm 1$ and is such that all smooth closed contractible $g _{Y}$-geodesics in $Y$ are constant. Then the following holds:
\begin{itemize}
\item $X$ does not admit a complete Riemann-Finsler metric with negative curvature. 
\item Moreover, $X$ does not admit
a complete Riemann-Finsler metric with
a unique and non-degenerate class $\beta $ geodesic string. \label{part_chipm1}

\end{itemize} 
\end{theorem}
Note that $\chi(Y) \neq 1$ is of course essential, as  the
trivial fibration $X \to \{pt\}$, with $X$ admitting
a complete negatively curved metric, will satisfy the
hypothesis. The condition that there is a fiber class $\beta
\in \pi _{1} ^{inc} (X)$ is also essential, for any vector
bundle over a manifold admitting a Riemannian metric of negative
curvature admits a metric of negative curvature, Anderson
~\cite{cite_AndersonNegativeCurvaturefibration}. 

Theorem \ref{corollary_product} gives one set of examples.
A further basic set of examples for the theorem is obtained by starting
with any homomorphism 
\begin{equation} \label{equation_phi}
\phi: \pi _{1} (Y, y _{0}) \to \operatorname
{Isom} (Z, g _{Z}),  \quad \text{(the group of all isometries).} 
\end{equation}
where $g _{Z}$ is a taut Riemannian metric, and there is a class $ \beta _{Z} \in \pi_{1}
^{inc}({Z})$ (for example $(Z,g _{Z})$ is a non-simply
connected complete hyperbolic surface).  Suppose further:
\begin{enumerate}
\item  The orbit 
$$O := \bigcup _{\gamma \in \pi _{1} (Y,
y _{0}) } {\phi} _{*} (\gamma ) (\beta _{Z})$$ is 
 finite. \label{item_finiteorbit}
	\item $Y$ is closed and connected.
	\item All contractible smooth closed $g _{Y}$ geodesics in $Y$ are
	constant.
\end{enumerate}
We have the obvious induced diagonal action $$\pi
_{1} (Y, y _{0}) \to \operatorname {Diff}  (Z \times
\widetilde{Y}), \text{ (the group of all
diffeomorphisms)},$$ $$\gamma \mapsto ((z, y) \mapsto (\phi
(\gamma) (z), \gamma \cdot y)), $$
for $\widetilde{Y} $ the universal
cover of $Y$.  Taking
the quotient of $Z \times \widetilde{Y} $ by this action, we get an associated ``flat'' bundle $Z
\hookrightarrow X _{\phi} \xrightarrow{p} Y$, with a metric
$g _{\phi }$
induced from the product metric $\widetilde{g} = g _{Z}
\oplus g _{Y}$, on the covering space $q: Z \times
\widetilde{Y} \to Z \times Y$. 
\begin{lemma} \label{lemma_betataut}
Let $p: (X _{\phi}, g _{\phi }) \to (Y, g _{Y})$  be as above, then this is
a $\beta $-taut fibration, where $\beta = i _{*} (\beta
_{Z}) $, for $i _{*}: \pi_{1} ^{inc}(Z) \to \pi_{1}
^{inc}(X _{\phi})$  induced by inclusion. 
\end{lemma}
By the lemma above, $p: (X _{\phi}, g _{\phi }) \to (Y, g _{Y})$  satisfies the hypothesis of the theorem above.
Yet more concretely:
\begin{example} \label{example_}
Suppose we have $\beta
_{Z} \in \pi_{1} ^{inc}({Z})$, and 
let $\phi: Z \to Z$ be an isometry of a 
taut metric $g _{Z}$.
Then by the construction above, the
mapping torus $$(Z, g _{Z}) \hookrightarrow  (X _{\phi},
g _{\phi})  \xrightarrow{\pi
_{} } S ^{1}$$  has the structure of a $\beta$-taut
fibration,
satisfying the hypothesis of the theorem, for $\beta
= i _{*} (\beta _{Z}) $ as above.
\end{example}
The next corollary of Theorem \ref{thm:Fibration} is
immediate.
\begin{corollary} \label{corollary_Thurston0}
Let $$(Z _{g, Z}) \hookrightarrow (X _{\phi}, g _{\phi}) \to
(Y, g _{Y})$$ be as in the
construction above for $Z, g _{Z}$ having non-positive
curvature, and let $ \beta _{Z} \in \pi_{1} ^{inc}({Z})$.
Then if $\chi(Y) \neq \pm 1 $:
\begin{enumerate}
\item $X _{\phi }$ does not admit a complete Riemann-Finsler metric with negative curvature.
\item Moreover, $X _{\phi}$ does not admit a Riemann-Finsler metric with
a unique and non-degenerate class $\beta$ geodesic string, for $\beta = i _{*} (\beta _{Z}) $ as above.
\end{enumerate}
As a special
case, this applies to the mapping tori $X _{\phi }$, for
$\phi: Z \to Z $ an isometry of a complete Riemannian
non-positively curved metric on $Z$, satisfying the
finiteness condition \ref{item_finiteorbit}. (The non-positive curvature hypothesis is for concreteness we may of course replace this condition by tautness.)
\end{corollary}
In the special case when $Z$ is compact, the first part of
the above corollary can be deduced, with some work, from Preissman's theorem (specifically, because of the condition
\ref{item_finiteorbit}), see also ~\cite[Theorem
9.3.4]{cite_BuragoLengthSpaces} for a generalization that
fits our Finsler setting. The second part is new even when
$Z$ is compact.
In the non-compact case, as far as I know, the above
corollary and the more basic Theorem
\ref{corollary_product}, are the only presently known
extensions of Preissman's theorem. 

\section{Proof of Theorem \ref{thm:notgeodesible}}
The first part of the theorem clearly follows by the second
part. So let $\{g _{t}\}$, $t \in [0,1]$  be as in the
hypothesis, with 
\begin{equation} \label{eq_mainC}
\sup _{t} |\max _{o \in S  (g _{t}, \beta )} l _{g _{t}} (o) -  \min
_{o \in S  (g _{t}, \beta )} l _{g _{t}} (o)| < C, 
\end{equation}
and suppose that 
\begin{equation*}
\sup _{(o,t) \in \mathcal{O} (\{g _{t}\},\beta )} l _{g _{t}} (o) = \infty.
\end{equation*}
Then we have a sequence $\{o _{k}\}$, $k \in \mathbb{N}
$, of closed class $\beta $ $g _{t _{k}}$-geodesics
in $X$, satisfying:
\begin{enumerate} 
	\item $\lim
_{k \to \infty } t _{k} = t _{\infty } \in [0,1]. $
\label{cond:1}
\item  $\lim _{k \to \infty }l _{g _{t _{k}}} (o _{k})
= \infty $,  where $l _{g _{t _{k}}} (o _{t _{k}})$ is
the length with respect to $g _{t _{k}}$. \label{cond:2}
\end{enumerate}
Let $o _{\infty}$ be a minimal,
class $\beta $, $g _{\infty
} = g _{t _{\infty }}$ geodesic in $X$. And let $L$ denote
its length $g _{\infty }$ length. Let $g _{aux}$ be
a fixed auxiliary metric on $X$, and let $L _{aux} $ be the
$g _{aux }$ length of $o _{\infty }$.

Define a pseudo-metric $d _{C_0}$ on the space of metrics on
$X$ as follows. Set $K = \image o _{\infty}$. And
set $$V \subset TX = \{v \in TX \,|\, \pi ({v})  \in
K \text{ for $\pi _{}: TX \to X$ the canonical
projection, and $|v| _{aux} = 1$} \},$$
where $|v| _{aux}$ is the norm taken with
respect to $g _{aux}$.

Then define: $$d _{C ^{0}} (g _{1}, g _{2}) = \sup _{v \in
V} | |v| _{g _{1}} - |v| _{g _{2}}|. $$  

By Properties \ref{cond:1} and \ref{cond:2} we may find
a  $k > 0$ such that: 
\begin{equation} \label{eq:lessepsilon}
d _{C ^{0}} (g _{t
_{k}}, g _{t _{\infty }}) < \epsilon 
\end{equation}
and 
\begin{equation} \label{eq_Laux}
l _{g _{t _{k}}}
(o _{k}) > C + L + L _{aux} \cdot \epsilon.
\end{equation}

By \eqref{eq:lessepsilon},
we have:
$$l _{g _{t _{k}}} (o
_{\infty}) < l _{g _{t _{\infty }}} (o _{\infty}) + L _{aux}
\cdot \epsilon = L + L _{aux} \cdot \epsilon.$$
Combining with \eqref{eq_Laux} we get:
\begin{equation*}
l _{g _{t _{k}}}
(o _{k}) > l _{g _{t _{k}}} (o _{\infty }) +C.
\end{equation*}
Since we may find a closed $g _{t _{k}}$-geodesic $o'$
satisfying $l _{g _{t _{k}}} (o') \leq l _{g _{t _{k}}} (o
_{\infty })$,
we get that
$$|\max _{o \in S  (g _{t _{k}}, \beta)} l _{g _{t _{k}}} (o)
-  \min _{o \in S  (g _{t _{k}}, \beta)} l _{g _{t _{k}}} (o)|> C,
$$
and so we are in contradiction. 

Thus, \begin{equation*}
\sup _{(o,t) \in \mathcal{O} (\{g _{t}\},\beta )} l _{g
_{t}} (o) < \infty.
\end{equation*}
It follows, by an analogue of Lemma \ref{lemma_containedinCompact}, 
that the images of all elements $o \in S 
(\{g  _{t} \}, {\beta} ) $ are contained
in a fixed compact ${T} \subset X$. Compactness of $S 
(\{g  _{t} \}, {\beta} ) $ then readily follows by the Arzella-Ascolli
theorem. 
\qed
\section{Proof of Lemma \ref{lemma_betataut}} \label{sec_Proof of Lemma lemma_betatau}
Let ${\phi}_*: \pi _{1} (Y, y _{0}) \to
\operatorname {Aut} (\pi_{1} ^{inc}(Z))  $ be the natural
induced action, where $\operatorname {Aut} (\pi_{1}
^{inc}(Z))$ denotes the group of set isomorphisms of	$\pi_{1}
^{inc}(Z))$. And such that the orbit $$O := \bigcup _{\gamma \in \pi _{1} (Y,
y _{0}) } {\phi} _{*} (\gamma ) (\beta _{Z})$$ is 
finite. 

As $g _{Z}$ is taut, $S (g _{Z}, {\phi}
_{*} (\gamma ) (\beta _{Z}))$ is compact for each $\gamma $, 
where $S (g _{Z}, {\phi}
_{*} (\gamma ) (\beta _{Z}))$ is the space of geodesics as
in Definition \ref{definition_betataut}. By the condition on contractible geodesics of $g _{Y}$, we
get: 
\begin{align*}
	S (g _{\phi }, \beta ) & = q_* (S (g _{Z} \oplus g _{Y},
	\beta)) ) \\
	& = \bigcup
_{\beta  \in O } q_*(S (g _{Z}, \beta ) \times \widetilde{Y}),  
\end{align*}
for $q _{*}: L(Z \times
\widetilde{Y}) \to L(Z \times Y)$ induced by the quotient map $q: Z \times
\widetilde{Y} \to Z \times Y$, (as in the preamble to the
statement of the lemma) and where $S (g
_{Z}, \gamma) \times \widetilde{Y}$ is understood as
a subset  $$S (g
_{Z}, \gamma) \times \widetilde{Y} \subset L (Z) \times
\widetilde{Y} \subset L (Z \times \widetilde{Y}).$$
Given that $O$ is finite, this then readily implies our claim. 
\qed
\section{Preliminaries on Reeb flow} \label{sec_Some preliminaries on Reeb dynamics}
Let $(C ^{2n+1}, \lambda ) $ be a contact manifold with
$\lambda$ a contact form, that is a one form s.t. $\lambda
\wedge (d \lambda) ^{n} \neq 0$.    Denote by
$R ^{\lambda} $ the Reeb vector field 
satisfying: $$ d\lambda (R ^{\lambda}, \cdot ) = 0, 
\quad \lambda (R ^{\lambda}) = 1.$$ 
Recall that a \textbf{\emph{closed $\lambda $-Reeb orbit}}
(or just Reeb orbit when $\lambda $ is implicit)  is
a smooth map $$o: (S ^{1} = \mathbb{R} / \mathbb{Z})    \to
C $$ such 
that $$ \dot o (t) = c R ^ {\lambda}  (o (t)), $$ 
with $\dot o (t) $ denoting the time derivative,
for some $c>0$ called period. Let $S (R ^{\lambda
}, \beta )$ denote the space of all closed ${\lambda }$-Reeb
orbits in free homotopy class $\beta $, with its compact
open topology. And set $$\mathcal{O} (R
^{\lambda}, \beta ) = S (R ^{\lambda }, \beta )/S ^{1}, $$
where $S ^{1} = \mathbb{R} ^{} /\mathbb{Z}  $ acts by reparametrization $t \cdot o(\tau)
= o (t + \tau)$.


\section{Definition of the functional $\operatorname
{F}$ and proofs of auxiliary results} \label{sec:Definition of F} Let $X$ be a 
manifold with a taut metric $g$.  Let $C$ be the unit cotangent bundle
of $X$, with its Louiville contact 1-form $\lambda _{g}$.
If $o: S ^{1} = \mathbb{R} ^{}/\mathbb{Z} \to X$ is a unit speed
closed geodesic, it has a canonical lift $\widetilde{o}: S ^{1} \to C$. 
If $ \beta \in \pi _{1} ^{inc}  (X)$, let $\widetilde{\beta} \in \pi _{1}  (C) $ denote class $[\widetilde{o} ] \in \pi
_{1} (C)$, where $o$ is a unit speed closed geodesic
representing $\beta $.

Let $S (R ^{\lambda _{g}}, \widetilde{\beta})$ be the orbit
space as in Section \ref{sec_Some preliminaries on Reeb
dynamics}, for the Reeb flow of the contact form $\lambda _{g}$. And set $$\mathcal{O} _{g,
\beta}  = \mathcal{O} (R ^{\lambda _{g}}, \widetilde{\beta
} ) :=  S (R ^{\lambda _{g}}, \widetilde{\beta})/S ^{1},$$  
i.e. this can be identified with the space of class ${\beta
} $ $g$-geodesic strings.
By the tautness assumptions $\mathcal{O} _ {g, \beta }$ is compact.

We then define $$F (g, \beta ) = i (\mathcal{O} _{g, \beta
}, R ^{\lambda _{g}}, \widetilde{\beta})
\in \mathbb{Q} $$ where the
right hand side is the Fuller index as outlined in the
Appendix \ref{appendix:Fuller}.
As a basic example we have:
\begin{lemma} \label{lemma_compact}
Suppose that $g$ is a complete metric on $X$, all
of whose class $\beta \in \pi _{1} ^{inc} (X)$ geodesics are
minimal, then $g$ is $\beta$-taut.
\end{lemma}
\begin{proof} [Proof]
First we state a more basic lemma.
\begin{lemma} \label{lemma_containedinCompact} Suppose that
$g$ is a complete metric on $X$, $\beta \in \pi _{1}  ^{inc}
(X)$ and let $S \subset L _{\beta} X$ be a subset on which
the $g$-length functional is bounded from above. Then the images in $X$ of
elements of $S$ are contained in a fixed compact subset of $X$.
\end{lemma}
\begin{proof} [Proof]
Suppose otherwise. Fix an exhaustion by
nested compact sets $$\bigcup _{i \in \mathbb{N} } K _{i}
= X, \quad K _{i} \supset
K _{i-1}.$$ 
Then either there is sequence $\{o _{i}\} _{i \in \mathbb{N}
}$, $o _{i} \in S$  s.t. $o _{i} \in K _{i} ^{c} $, for $K _{i} ^{c}$ the
complement of $K _{i}$, which contradicts the fact that
$\beta $ is end incompressible. 
Or there is a sequence $\{o _{k} \} _{k \in \mathbb{N} }$,
$o _{k} \in S$ s.t.:
\begin{enumerate}
\item Each $o _{k}$
intersects $K _{i _{0}}$ for some $i _{0}$ fixed.
\label{condition_intersects}
\item For each $i \in \mathbb{N} $
there is a $k _{i} >i$ s.t. $o _{k _{i}}$ is
not contained in $K _{i}$. \label{condition_notcontained}
\end{enumerate}
Now if $\operatorname {diam} (o _{k})$ is bounded in $k$, then condition
\ref{condition_intersects} implies that $o _{k}$ are
contained in a set of bounded diameter. (Here $\operatorname
{diam} (o _{k})$ denotes the diameter of $\image o _{k}$.) Consequently, by
Hopf-Rinow theorem ~\cite{cite_ChernBook},  $o _{k}$ are contained in a compact set.
But this contradicts condition \ref{condition_notcontained},
and the fact that $K _{i}$ form an exhaustion of $X$. 

Thus, we conclude that $\operatorname {diam} (o _{k})$ is
unbounded, but this contradicts the hypothesis.
\end{proof}

Returning to the proof of the main lemma.
By assumption, closed, class $\beta \in \pi _{1}  ^{inc}
(X)$ geodesics are $g$-minimizing in their homotopy class and in particular have
fixed length.  By the lemma above there is a fixed $K \subset X$ s.t. every class
$\beta $ closed geodesic has image contained in $K$. Then
compactness of $S (g, \beta )$ follows by Arzella-Ascolli theorem.

\end{proof}

\begin{proof} [Proof Theorem \ref{thm:invariantF}]
Let $\beta \in \pi _{1} ^{inc} (X)$, be given and let $g$ be
$\beta$-taut. We just need to prove that $F (g, \beta)$ is
invariant under a $\beta $-taut deformation of $g$.  So let $\{g _{t}\} $, $t \in
[0,1]$ be a $\beta $-taut deformation of metrics on
a compact manifold $X$. Let $R  ^{\lambda _{g _{t}}} $ be the geodesic flow on
the $g _{t}$ unit cotangent bundle $C _{t}$. Trivializing
the family $\{C _{t}\}$ we get a family $\{R _{t}\}$ of
flows on $C \simeq C _{t} $, with $R _{t}$ conjugate to $R
^{\lambda  _{g _{t}}}$.

Let $\mathcal{O}
(\{R _{t} \}, \widetilde{\beta}) $ be the
cobordism as in \eqref{equation_Ohomotopy}, where $ \widetilde{\beta} \in \pi _{1} (C)$
is as above. Then $\mathcal{O}
(\{R _{t} \}, \widetilde{\beta}) $ is compact as
$S (\{g _{t}\}, \beta )$ is compact by assumption.

Basic invariance of the Fuller index, that is
\eqref{eq_basicinvariance}, immediately yields:
  $F (g _{0},
\beta) = F (g _{1}, \beta )$. 
\end{proof}
\section{Proof of Theorem \ref{thm_estimateIntro}} \label{sec_Proof of Theorem thm_estimategeodesics}
We already know by Example
\ref{exm_negativecurvatureFuller} that $F (g _{0}, \beta)
= \frac{1}{k}$ and hence by Theorem \ref{thm:invariantF}
$F (g _{1}, \beta) = \frac{1}{k}$.
Let $U$ denote the open subset of $L _{\beta} X$ consisting
of loops with $g _{1}$-length less then $C _{\beta}$.
By ~\cite[Lemma 4.1]{cite_SavelyevFuller}, 
for all $\epsilon >0$ sufficiently small, for any $g'$, $C
^{0}$ $\epsilon $ close to $g$ the following holds.
Set $g' _{t} = (t-1) \cdot g _{1} + t \cdot g'$, for
$t \in [0,1]$, then $$ \widetilde{N} = \mathcal{O} (\{g'
_{t}\}, \beta) \cap (U \times [0,1])$$ is an open and compact subset of $\mathcal{O} (\{g' _{t}\}, \beta)$. 	

Now set $$N _{1} = \widetilde{N} \cap (L _{\beta} X \times \{1\}), $$ and $N _{0} = \mathcal{O}
(g _{1}, \beta)$. 
By the invariance property \eqref{eq_basicinvariance} of the Fuller index, we then have
that $$\frac{1}{k} = i (N _{0}, R ^{\lambda _{g _{1}}})
= i (N _{1}, R ^{\lambda _{g'}}). $$
On the other hand, by construction and by index computations
as in ~\cite[Section 2]{cite_SavelyevFuller}),
we get:
$$i (N _{1}, R ^{\lambda _{g'}}) =  \displaystyle \sum_{o \in
\mathcal{O} (g',\beta) \cap U} \frac{(-1) ^{\operatorname {morse}  (o)}} {\mult (o)}.$$
If $\epsilon $ is chosen to be sufficiently small then
$\mathcal{O} (g',\beta) \cap U = \mathcal{O} _{C _{\beta}} (g', \beta
)$. So that we are done. 
\qed
\section{Proof of Theorem \ref{thm_IntroTn}} \label{sec_Proof of Theorem thm_IntroTn}
Let $\beta$ be a non-zero class, as in the statement. Let $Y
\subset T ^{n}$ be a totally geodesic submanifold
diffeomorphic to $T ^{n-1}$, containing $\image \beta$. Let
$\alpha \in H ^{1} (T ^{n}, \mathbb{Z})$ be the Poincare dual of $Y$ and
let $p: T ^{n} \to S ^{1}$ be the classifying map of
$\alpha$. Then clearly $p$ is a $\beta$-taut fibration.

By Theorem \ref{theorem_Eulercharacteristic}, $F (g _{0}, \beta) =0$. 
By the latter
proof, for all $\epsilon >0$  sufficiently small, for any $g'$ as in the statement of our theorem we have:
\begin{equation*}
\displaystyle \sum_{o \in
\mathcal{O} _{C _{\beta}} (g',\beta)} \frac{(-1) ^{\operatorname
{morse}  (o)}} {\mult (o)} = 0.
\end{equation*}
The conclusion then readily follows by basic arithmetic.
\qed

\section{Proof of Theorem \ref{theorem_Eulercharacteristic}} \label{sec_Proof of Theorem theorem_Eulercharacteristic}
This is an application of
Morse theory. As $g$ is $\beta $-taut, $S (g, \beta)$ is
compact.
Let $$L = \sup _{o \in S (g, \beta )} \energy
_{g}(o), $$ where 
$$\energy _{g}: {L}  _{\beta } X \to \mathbb{R} ^{}, $$  is
the function:
\begin{equation} \label{equation_energyfunction}
 \energy _{g} (o) = \int_{S ^{1}}^{} \langle \dot o (t), \dot
 o (t) \rangle _{g} dt.
\end{equation}
Choose $C >L$ and let $U$ denote the subspace of $L
_{\beta } X$ consisting of loops with $g$-energy less than $C$.
Now $U$ has the homotopy type of $L _{\beta}
X $. This can be proved without infinite dimensional Morse
theory. We may use the finite dimensional broken
geodesic approximation as in
Milnor~\cite{cite_MilnorMorsetheory}, (passing to the limit) and the fact that there are no geodesics in the complement of $U$.

If $g'$ is sufficiently $C ^{0}$ nearby to $g$ and is $\beta
$-regular than $$F (g, \beta ) = \displaystyle \sum_{o \in
\mathcal{O} (g',\beta) \cap U} (-1) ^{\operatorname
{morse}  (o)}. $$  

The latter assertion is shown similarly to the proof of Theorem
\ref{thm_estimateIntro}, except now there is no multiplicity
weight since our geodesics are forced have multiplicity one,
by the condition that $\beta $ is not a power.
To finish the proof we just need to show that $\sum_{o \in
\mathcal{O} (g',\beta) \cap U} (-1) ^{\operatorname
{morse}  (o)}$ is the Euler characteristic of $U/S ^{1}$, since the latter is the Euler characteristic of $L _{\beta}X/S ^{1}$.

Let us now denote by  $\mathcal{L}  _{\beta } X$ the Hilbert
manifold of $H ^{1}$ loops, in class $\beta $, 
as used for example in the classical work of
Gromoll-Meyer~\cite{cite_GromollMeyerGeodesics}. 
We also denote by $\mathcal{U}$ the $C$-sublevel
set analogous to $U$.
The Hilbert manifold $\mathcal{L}  _{\beta } X$ is well known to be homotopy equivalent to $L
_{\beta } X$ with its previously used compact open topology.

The energy function $\operatorname{energy} _{g'}:
\mathcal{L}  _{\beta } X \to \mathbb{R} $, defined as above,
is smooth, $S ^{1}$ invariant and satisfies the Palais-Smale condition.  The
flow for its negative gradient vector field $V$ is complete,
and we can do Morse theory mostly as usual. This
is understood starting with the work of
Klingenberg~\cite{cite_KlingenbergClosedGeodesics}, with the
framework of Palais and Smale ~\cite{cite_PalaisSmale}. Note
that all this also applies to $\mathcal{U}$.
In our case, $\operatorname{energy} _{g'}$ is moreover a Morse-Bott function with critical manifolds $C _{o}$ corresponding to $S ^{1}$ families of geodesics, for each geodesic string $o$.

There is an induced Morse-Bott cell decomposition on
$\mathcal{U} $, meaning a stratification formed by $V$ unstable
manifolds of the above mentioned critical manifolds $C
_{o}$. This is Bott's extension of the fundamental Morse
decomposition theorem.  Now the $S ^{1}$ action on $\mathcal{L} 
_{\beta}X$ is free by the condition that
$\beta $ is not a power. This action is not smooth, but it
is continuous. So taking the topological 
$S ^{1}$ quotient, we get a CW cell
decomposition of $\mathcal{U}/S ^{1}$, with one $k$-cell for
each closed $g'$-geodesic string $o$ in $\mathcal{U}$, with Morse index $\operatorname
{morse} (o) = k$. (Here the Morse index is the
Morse-Bott index of the critical manifold $C _{o}$.) All of
the above is well understood, see for instance ~\cite{cite_GromollMeyerGeodesics}.

From the above cell decomposition, we readily get that the
homology $$H_*(\mathcal{U}/S ^{1}, \mathbb{Z})
= H_*(U/S ^{1}, \mathbb{Z}) = H_*({L}
_{\beta}X/S ^{1}, \mathbb{Z}) = H _{*} ^{S ^{1}} (L
_{\beta}  X, \mathbb{Z})$$ is finite dimensional. 
And we get that: 
\begin{align*} 
\chi(U/S ^{1}) & = \displaystyle \sum_{o \in
 \mathcal{O} (g',\beta) \cap U} (-1) ^{\operatorname
{morse}  (o)}  \quad \text{(immediate from the
cell decomposition)}. 
\end{align*}
\qed
\section{Proof of Theorem \ref{thm:Fibration} and
its corollaries}
\label{sec:Proof of Theorem fibration}
We first prove:
\begin{theorem} \label{thm:EulerProduct} Let $p: X \to Y$ be
a $\beta $-taut fibration as in the statement of Theorem
\ref{thm:Fibration} and $ \beta \in \pi _{1} ^{inc} (X)$ a fiber class.
Then 
\begin{equation} \label{equation_mainformula}
F (g, {\beta}) = card \cdot \chi (Y) \cdot
F (g _{Z}, \beta _{Z}),
\end{equation}
where $card \in \mathbb{N} - \{0\} $ is the cardinality of a certain
orbit of the holonomy group (as
explained in the proof), and where $\beta _{Z}$ is as in
Lemma \ref{lemma_betataut}.
\end{theorem}
\begin{proof} [Proof]
We have a natural subset of $\mathcal{O}' \subset \mathcal{O} {g, \beta }$, consisting of all vertical geodesics, that is
$g$-geodesics contained in fibers $p ^{-1} (y)
= Z _{y} $. In fact, 
\begin{equation} \label{eq_O'=O}
\mathcal{O}' = \mathcal{O} {g,
\beta},
\end{equation}
for if $o$ is any class $\beta $ closed geodesic, the
projection $p (o)$ is a contractible closed $g
_{Y}$-geodesic, and by assumptions is constant.

In particular, there a natural continuous projection
$$\widetilde{p}: \mathcal{O} {g, \beta} \to Y,
\quad \widetilde{p}(o)
= y$$ where $y$ is determined by the condition that $$Z _{y}
\supset \image o.$$
We will use this to construct a suitable (in a sense
abstract i.e. not Reeb) perturbation of the vector field
$R ^{\lambda _{g}}$, using which we can calculate the
invariant $F (g, \beta )$.

Fix a Morse function on $f$ on $Y$, let $C= S ^{*} X$ denote
the $g$-unit cotangent bundle of $X$. For $v \in T _{x} X$ let
$\langle v|  $ denote the
functional $$T _{x} X \to \mathbb{R} ^{}, \quad  w
\mapsto \langle v, w \rangle _{g}.$$ 
Define $\widetilde{f}: C \to \mathbb{R} ^{} $ by
$$\widetilde{f} (\langle v|) := f (p (v)),$$ also define
$$P: C \to \mathbb{R} ^{} $$ by $$P (\langle v|) := |P^{vert} (v)|  ^{2}
_{g},$$ where $P ^{vert} (v)$ denotes the	$g$-orthogonal
projection of $v$ onto the $T ^{vert} _{x} X \subset
T _{x}X$, for $T ^{vert} X$ the vertical tangent bundle of
$X$, i.e. the kernel of the bundle map $p_*: TX \to TY$.

Next define $F: C \to \mathbb{R} ^{} $ by:
\begin{equation*}
F (\langle v|) :=  P (\langle v|) + \widetilde{f}  (\langle v|).
\end{equation*}

Set $$V _{t} = R ^{\lambda _{g}} - t\operatorname {grad} _{g
_{S}}
F,$$ where the gradient is taken with
respect to the Sasaki metric $g _{S}$ on $C$
~\cite{cite_SasakiMetric} induced by $g$. The latter Sasaki
metric is the natural metric for which we have an orthogonal splitting $TC
= T ^{vert}C \oplus T ^{hor} C$, where $T ^{vert} C$ is the
kernel of $pr_*: TC \to TX$, induced by the natural
projection $pr: C \to X$, and where $T ^{hor} C$  is the
$g$ Levi-Civita horizontal sub-bundle.  

Set $\mathcal{O} _{t} = \mathcal{O} (V _{t},
\widetilde{\beta} )$, where	$\widetilde{\beta } $  is
as in Section \ref{sec:Definition of F}.
\begin{lemma} \label{lemma_epsilon}
We have:
\begin{enumerate}
\item For all $t \in [0, 1]$, $N_t := \mathcal{O} _{t}  \cap
\mathcal{O} _{g, \beta  }$ is open and closed in $\mathcal{O} _{t}$.
\item For all $t \in (0, 1]$, $N _{t} = \cup _{y \in \operatorname {crit}  (f)} \widetilde{p}
^{-1} (y),$ where $\operatorname {crit}  (f)$ is the set of critical points of $f$. 
\end{enumerate}
\end{lemma}
\begin{proof} [Proof]
It is easy to see that $V _{t}$ is complete and without
zeros. Suppose that $t>0$.
Let $\langle v _{\tau} | $,
$\tau \in \mathbb{R} ^{} $ be the flow line of $V _{t}$,
through $\langle v _{0} | $, i.e. $\langle v _{\tau} | = \phi 
_{\tau} (\langle v _{0} | )$, for $\phi  _{\tau}$ the
time $\tau$ flow map of $V _{t}$. By the fact that the fibers of $p$ are
assumed to be parallel, we have that 
$${R ^{\lambda _{g}}}
(P) =0, \quad \text{using the derivation notation}.$$ 
Also, $$\operatorname {grad} _{g
_{S}} \widetilde{f} (P) =0,$$ which readily follows by the
conjunction of $g _{S}$ being Sasaki and the fibers of $p$
being parallel.
Consequently, the function $$\tau \mapsto P (\langle v _{\tau}|)
= |P^{vert} (v
_{\tau})| ^{2} _{g}$$ is monotonically decreasing unless either:
\begin{enumerate}
	\item $v _{0}$ is tangent to $T ^{vert} X$, in which case
	for all $\tau$, $v _{\tau} $ are tangent to $T ^{vert} X$
	and $|P^{vert} (v _{\tau})| ^{2} _{g} =1$.
	\item For all $\tau$, $|P^{vert} (v _{\tau})| ^{2} _{g} =0$.
\end{enumerate}

In particular, the closed orbits of $V _{t}$ split into two
types. 
\begin{enumerate}
	\item Closed orbits $o (\tau) = \langle v _{\tau}
	|  $ with $v _{\tau}$
always tangent to $T ^{vert}X$. In this case we may
immediately, conclude that $o$ is a lift to $C$ of a closed
$g$-geodesic contained in the fiber over a critical point of $f$.
\item Closed orbits $o (\tau) = \langle v _{\tau} |  $ for
which $v _{\tau}$
is always $g$-orthogonal to $T ^{vert}X$. 
\end{enumerate}

Clearly, the conclusion follows.
\end{proof}
\begin{remark} \label{remark_}
It would be very fruitful to remove the condition on the
fibers of $p$ being parallel. But our argument would need to
substantially change.
\end{remark}

We return to the proof of the theorem.
Set $$\widetilde{N}   = \{(o,t) \in L _{\widetilde{\beta} } C \times [0, \epsilon]
\,|\, o \in {N} _{t}  \},$$ where $L _{\widetilde{\beta}
} C$ denotes the $\widetilde{\beta } $ component of the free
loop space as previously.
By part I of Lemma \ref{lemma_epsilon},  this is an open compact subset of $\mathcal{O} (\{V
_{t}\}, \widetilde{\beta })$ s.t. $$\widetilde{N}
\cap  (L _{\widetilde{\beta} } C \times \{0\}) = \mathcal{O} (R ^{\lambda
_{g}}, \widetilde{\beta}),$$ (equalities throughout are up to natural set theoretic identifications.)


By definitions: 
\begin{equation*}
N _{t}  = \widetilde{N}  \cap (L _{\widetilde{\beta}
} C \times \{t\}). 
\end{equation*}
Now  the basic invariance of the Fuller index, \eqref{eq_basicinvariance} gives:
\begin{equation*}
i(N _{0}, R ^{\lambda _{g}}, \widetilde{\beta }) = i (N
_{1}, V _{1}, \widetilde{\beta } ).
\end{equation*}

We proceed to compute the right hand side.
Fix any smooth Ehresmann connection $\mathcal{A} $ on the fiber bundle $p:
X \to Y$. This induces a holonomy homomorphism:
$$hol _{y}: \pi _{1} (Y, y) \to \operatorname {Aut}	\pi _{1} (Z
_{y}) \text{ (the right-hand side is the group of set
automorphisms)}, $$ with image denoted $\mathcal{H} _{y} \subset \operatorname {Aut}	\pi _{1} (Z
_{y})$.

Let $\beta _{Z}$ denote a class in $\pi _{1} (Z _{y})$ s.t. $(i
_{Z _{y}})_* (\beta _{Z}) = \beta $, for $i _{Z _{y}}:
Z _{y} \to X$ the inclusion map. 
Set $$S _{y} := \bigcup _{g \in \mathcal{H} _{y}} g ( \beta _{Z})
\subset \pi _{1}  (Z _{y}). $$ 
Then for another $y' \in Y$, 
\begin{equation} \label{equation_Sy}
h_*: S _{y'} \to S _{y}, 
\end{equation}
is an isomorphism, 
where $h: Z _{y} \to Z _{y'}$ is a smooth map given by the
$\mathcal{A} $-holonomy map determined by some path from $y$
to $y'$, and where $h _{*}$ is the naturally induced map.

Denoting by $g _{y}$ the restriction of $g$ to the fiber $Z _{y}
$, let $R ^{{y}}$ denote the $\lambda _{g _{{y}}}$ Reeb vector field on the $g _{Z
_{y}}$-unit cotangent bundle $C _{y}$ of $Z _{y}$. 
The cardinality $card$ of $S _{y}$ is finite, as otherwise
we get a contradiction to the compactness of $S (g,
{\beta} )$.
Now
\begin{equation*}
\widetilde{p} ^{-1} (y) = \bigcup _{\alpha \in S _{y}}
\mathcal{O} (R ^{\lambda _{y}}, \alpha).
\end{equation*}

From part 2 of Lemma \ref{lemma_epsilon} and by
index computations as in ~\cite[Section
2]{cite_SavelyevFuller}), we get: 
\begin{equation*}
i (N _{1}, V _{1}, \widetilde{\beta } ) = \sum_{y
\in \operatorname {crit} (f)} (-1)^{\operatorname {morse}
(y)} \cdot i (\widetilde{p} ^{-1} (y), R ^{\lambda _{y}},
\widetilde{\beta }  ), 
\end{equation*}
where $\operatorname {morse} (y)$ denotes the Morse index of
$y$.
Now
\begin{align*}
i (\widetilde{p} ^{-1} (y), R ^{\lambda _{y}},
\widetilde{\beta }) & = \sum _{\alpha \in S _{y}} 
i (\mathcal{O} (R ^{\lambda _{y}}), R ^{\lambda _{y}}, 
{\widetilde{\alpha}}) \\ 
& = \sum_{\alpha \in S _{y}} F (g _{y}, \alpha). \\
& = card \cdot F (g _{Z}, \beta _{Z}),
\end{align*}
where the last equality follows by \eqref{equation_Sy}, and
by the condition \ref{item_tauthomotopyfibers} in the
Definition \ref{definition_betataut}. And so the result follows. 
\end{proof}
We return to the proof of the main theorem. The first part
immediately follows from the second, as 
any class $\beta \in \pi _{1} ^{inc} (X)$ geodesic strings of
a complete negatively curved Riemannian manifold $X$ are unique. We prove the
second part first. Suppose first that $\beta
$ is an $n$-power: $\beta = \alpha ^{n}$, for some
$n \geq 1$ where $\alpha $ is not a power. By the assumption that
all contractible $g _{Y}$ geodesics are constant, classical 
Morse theory Milnor~\cite{cite_MilnorMorsetheory} tells us
that $Y$ has vanishing higher homotopy groups $\pi _{k} (Y,
y _{0})$, $k \geq 2$. And in particular $i _{Z, *}: \pi
_{1} (Z, p _{0}) \to \pi _{1} (X, p _{0}) $ is a group
injection, by the long exact sequence of a fibration. It follows that $\alpha \in \pi _{1} ^{inc} (X)$
is also a fiber class. 

Now, any $\alpha $-class $g$-geodesic string must be contained in a fiber of $p$. For otherwise
we may find a $\beta $ class $g$-geodesic string, which is
not contained in a fiber of $p$, which would contradict 
\eqref{eq_O'=O}.
It readily follows that $p: X \to Y$ is also $\alpha $-taut.

Now, if $\chi(Y) \neq \pm 1$ then by \eqref{equation_mainformula} $F (g, \alpha
) \neq  1$, since $F (g _{Z}, \alpha _{Z})$ is an integer by
Theorem \ref{theorem_Eulercharacteristic}.  By Theorem \ref{theorem_Eulercharacteristic}
\begin{equation*}
F (g, \alpha ) = \chi ^{S ^{1}}(L	_{\alpha}X).
\end{equation*}
So if $X$ admits a complete metric with a unique and non-degenerate
class $\alpha$ $g$-geodesic string then we have: $$F (g, \alpha
) = \chi ^{S ^{1}}(L	_{\alpha}X) = 1$$ (see Proof
of Theorem \ref{theorem_Eulercharacteristic}), which is
impossible.
It follows that $X$ does not admit a metric with a unique and non-degenerate
class $\beta $ $g$-geodesic string. For if
$o,o'$ are distinct, non-degenerate, class $\alpha
$ $g$-geodesics strings, then
the $n$-fold covers $o ^{n}, (o') ^{n}$ are class $\beta
$, distinct non-degenerate $g$-geodesic strings.

We now prove the general case. Suppose by contradiction that $X$
admits a metric with a unique and non-degenerate class
$\beta $ $g$-geodesic string $o$. By the above, $\beta$ is not
atomic. Now $o$ covers a multiplicity one geodesic string
$\widetilde{o} $ in some class $\widetilde{\beta} \in \pi _{1}
^{inc} (X)$. Moreover, $\widetilde{o}$ is the unique
geodesic string in its class, otherwise $o$ would not be
unique in its class.
We prove that $\widetilde{\beta} $ is not
a power, which will be a contradiction to $\beta $ not being
atomic and will complete the
proof.

Suppose otherwise, so that  $\widetilde{\beta} _{x
_{0}}  = \alpha ^{k}$ for $k>1$ and $\alpha \in \pi _{1} (X,
x _{0})$, ($\beta _{x _{0}}$ is as in Definition
\ref{definition_indivisible}). 
Let $u$ be a class $\alpha $, closed
$g$-geodesic string (where $\alpha$ also denotes the class
in $\pi _{1}  ^{inc} (X)$ corresponding to the based class $\alpha $.) 
It is immediate
that the $k$ cover of $u$, $u ^{k}$ represents
$\widetilde{\beta} $ and is a $g$-geodesic string. By the
uniqueness, $\widetilde{o} = u ^{k}$. But this contradicts
simplicity of $\widetilde{o} $. So $\widetilde{\beta } $ is
not a power.

\qed
\begin{proof} [Proof of Theorem \ref{corollary_product}]
Let $g _{Z}, g _{Y}$ be complete
Riemannian metrics on $Z$ respectively $Y$ with non-positive curvature.  Take the product metric $g = g _{Z} \times
g _{Y}$ on $X = Z \times Y$.  For a class $\beta \in \pi _{1} (Z)$ in the image of the inclusion $\pi _{1}
(Z) \to \pi _{1} (X)$, the natural projection $X \to Y$ is
automatically a $\beta $-taut fibration. Then the
conclusion readily follows from Theorem \ref{thm:Fibration}.

\end{proof}

\begin{proof} [Proof of Theorem \ref{theorem_valuesOfInvariant}]
By Theorem \ref{thm:EulerProduct} 0 is certainly a value of
the invariant $F$. We first prove that every negative rational number is
the value of the invariant. 
Let $p,q $ be positive integers. Let $Y
$ be a closed surface of genus $(p +1) >1$ with a hyperbolic
metric $g _{Y}$, let $Z$ be the
genus 2 closed surface with a hyperbolic metric $g _{Z}$ and
let $\beta _{Z} \in \pi _{1} ^{inc} (Z)$ be
the class represented by a $2 \cdot q$-fold covering of
a simple closed loop representing a generator of the fundamental group of $Z$.

Let $X = Y \times Z$ with the product metric $g = g _{Y}
\times g _{Z}$ and $p: X \to Y $  the
canonical projection. By Theorem \ref{thm:EulerProduct} $$F
(g, \beta) = \chi(Y)  \cdot F (g _{Z}, \beta _{Z}) = (
- 2p) \cdot \frac{1}{2q} = - \frac{p}{q}, $$
where $\beta $ is  as in Lemma \ref{lemma_betataut}. So we
proved our first claim.

Let again $p,q$ be positive integers. Let $Y$ be closed surface of
genus $2$, with a hyperbolic metric $g _{Y}$. And let $Z$ be a manifold satisfying $F (g _{Z}, \beta
_{Z}) = -\frac{p}{2q}$ for some $\beta _{Z}$-taut metric $g
_{Z}$ on $Z$ and for some class $\beta _{Z} \in \pi _{1} ^{inc} (Z)$. This
exist by the discussion above. Let $g = g _{Y} \times
g _{Z}$ be the product metric on $Y \times Z$, and $\beta
 $ as above. Analogously to the discussion above we get:
$$F (g, \beta) = \chi(Y)  \cdot F (g _{Z}, \beta _{Z}) = (
-2) \cdot \frac{-p}{2q} =  \frac{p}{q}. $$
\end{proof}

\begin{appendices} 
\section {Fuller index and sky catastrophes}
\label{appendix:Fuller} Let $X$ be a complete vector field
without zeros on a manifold $M$. 
Set 
\begin{equation}  S (X, \beta) = 
   \{o \in L _{\beta} M \,\left .  \right | \exists p \in 
   (0, \infty), \, \text{ $o: \mathbb{R}/\mathbb{Z} \to M $ is a
   periodic orbit of $p X $} \}.
\end{equation}
The above $p$ is uniquely
determined and we denote it by $p (o) $ called the period of
$o$.

There is a natural $S ^{1}$ reparametrization action on $S (X, \beta)$:  $t \cdot o$ is the loop $t \cdot o(\tau) = o (t + \tau)  $.
The elements of $\mathcal{O} (X, \beta) := S (X, \beta)/S ^{1} $ will be 
called \textbf{\emph{orbit strings}}. 
Slightly abusing notation we just write $o$  for the
equivalence class of $o$.  

The multiplicity $m (o)$ of an orbit string is
the ratio $p (o) /l$ for $l>0$ the period of a simple
orbit string covered by $o$. 

We want a kind of fixed point index which counts orbit
strings $o$ with certain weights.
Assume for simplicity that  $N \subset \mathcal{O} (X, \beta) $ is
finite. (Otherwise, for a general open compact $N \subset
\mathcal{O} (X, \beta)$, we need to perturb.) Then
to such an $(N,X, \beta)$
Fuller associates an index: 
\begin{equation*}
   i (N,X, \beta) = \sum _{o \in N}\frac{1}{m
   (o)} i (o),
\end{equation*}
 where $i (o)$ is the fixed point index of the time $p (o) $ return map of the flow of $X$ with respect to
a local surface of section in $M$ transverse to the image of $o$. 

Fuller then shows that $i (N, X, \beta )$ has the following invariance property.
For a continuous homotopy $\{X _{t}
\}$,  $t \in [0,1]$ set \begin{equation*}  S (\{X _{t}\}, \beta)  = 
   \{(o, t) \in L _{\beta} M \times [0,1] \,|\,  
	  \text{$o \in S (X _{t})$}\}.
\end{equation*}
And given a continuous homotopy $\{X _{t}
\}$, $X _{0} =X $, $t \in [0,1]$, suppose that 
$\widetilde{N} $ is an open compact subset of 
\begin{equation} \label{equation_Ohomotopy}
\mathcal{O} (\{X _{t}\}, \beta  ) := S (\{X _{t} \}, \beta ) / S ^{1},
\end{equation}
such that $$\widetilde{N} 
\cap \left (L _{\beta }M \times \{0\} 
\right) / S ^{1} =N.
$$ Then if $$N_1 = \widetilde{N} \cap \left (L _{\beta }M 
  \times \{1\} \right) / S 
^{1}$$ we have 
\begin{equation} \label{eq_basicinvariance}
i (N, X, \beta ) = i (N_1, X_1, \beta).
\end{equation} 
We call this \textbf{\emph{basic invariance}}.  
In the case $\mathcal{O} (X
_{0}, \beta )$ is compact, $\mathcal{O} (X _{1}, \beta )$ is
compact for any sufficiently $C	^{0}$ nearby $X _{1}$, and
in this case basic invariance implies (see for instance
~\cite [Proof of Lemma 1.6]{cite_SavelyevFuller}):
\begin{equation} \label{eq_basicinvariance2}
i (\mathcal{O} (X _{0}, \beta ), X, \beta ) = i (\mathcal{O}
(X _{1}, \beta ), X_1, \beta).
\end{equation}

\subsection {Blue sky catastrophes} \label{appendix_bluesky}
\begin{definition} [Preliminary] 
A \textbf{\emph{sky catastrophe}} for a smooth family $\{X
_{t} \}$, $t \in [0,1]$, of non-vanishing vector fields on
a closed manifold $M$ is a continuous family of closed orbit
strings $\tau \mapsto o _{t _{\tau} }$, $o _{t _{\tau}
} $ is an orbit string of $X _{t _{\tau} } $, $\tau \in [0,
\infty)$, such that the period of $o _{t _{\tau} } $ is unbounded from above.
\end{definition}
A sky catastrophe as above was initially constructed by
Fuller ~\cite{cite_FullerBlueSky}. Or rather his
construction essentially contained this phenomenon. 
A more general definition appears in
~\cite{cite_SavelyevFuller}, we slightly extend it here to
the case of non-compact manifolds. All these
definitions become equivalent given certain regularity
conditions on the family $\{X _{t}\}$ and assuming $M$ is
compact.
\begin{definition}\label{def:bluesky}
Let  $\{X _{t} \}$, $t \in [0,1] $ be a continuous family of
non-zero, complete smooth vector fields on a manifold
$M$ and $\beta \in \pi _{1}  ^{inc} (X)$. 

We say that $\{X _{t}\}$  has a \textbf{\emph{catastrophe in class $\beta$}}, if
there is an element $$y \in \mathcal{O}  (X _{0}, \beta)
\sqcup \mathcal{O}  (X _{1}, \beta)
\subset \mathcal{O}  (\{X
_{t}\},  \beta)
   $$ such that there is no open compact subset
of $\mathcal{O} (\{X _{t}\}, \beta) $ containing $y$. 
\end{definition}
A vector field $X$ on $M$ is \textbf{\emph{geodesible}} if
there exists a metric $g$ on $M$ s.t. every flow line of $X$
is a unit speed $g$-geodesic. A family $\{X _{t}\}$ is
\textbf{\emph{geodesible}}  if there is a continuous family
$\{g _{t}\}$ of metrics, with $X _{t}$ geodesible with
respect to $g _{t}$ for each $t$.
A family $\{X _{t}\}$ is
\textbf{\emph{geodesible}}  if there is a continuous family
$\{g _{t}\}$ of metrics with $X _{t}$ geodesible with
respect to $g _{t}$ for each $t$. 
A \textbf{\emph{geodesible sky catastrophe}} is a geodesible
family $\{X _{t}\}$ with a sky catastrophe.
A \textbf{\emph{Reeb sky catastrophe}} is
a  family of Reeb vector fields $\{X _{t}\}$ with a sky catastrophe.
\end{appendices}
\section{Acknowledgements} 
This paper was substantially edited during my stay as
a member at IAS, and I am grateful for the wonderful time,
and the resources provided. 
\bibliographystyle{siam} 
\bibliography{link.bib} 
\end{document}